\documentclass[reqno]{amsart}
\usepackage{amsmath}
\usepackage{amsthm}
\usepackage{graphicx}
\usepackage{amsfonts}
\usepackage{amssymb}
\numberwithin{equation}{section}

\newtheorem{theorem}{Theorem}[section]

\newtheorem{proposition}[theorem]{Proposition}
\newtheorem{corollary}[theorem]{Corollary}

\theoremstyle{definition}

\newtheorem{definition}[theorem]{Definition}
\newtheorem{remark}[theorem]{Remark}

\newtheorem*{assumption}{Standing assumptions}

\newtheorem{innercustomthm}{MFG problem}
\newenvironment{problem}[1]
  {\renewcommand\theinnercustomthm{#1}\innercustomthm}
  {\endinnercustomthm}

\def\E{{\mathbb E}}
\def\R{{\mathbb R}}

\def\P{{\mathcal P}}

\def\B{{\mathcal B}}

\def\X{{\mathcal X}}

\def\L{{\mathcal L}}

\def\W{{\mathcal W}}

\def\F{{\mathcal F}}

\def\C{{\mathcal C}}

\title{Translation invariant mean field games with common noise}
\author{Daniel Lacker and Kevin Webster}

\begin{document}
\pagestyle{empty}

\begin{abstract}
This note highlights a special class of mean field games in which the coefficients satisfy a convolution-type structural condition. A mean field game of this type with common noise is related to a certain mean field game without common noise by a simple transformation, which permits a tractable construction of a solution of the problem with common noise from a solution of the problem without.
\end{abstract}

\maketitle

\section{Introduction}
The goal of this paper is to demonstrate how a typical structural property can be exploited to construct a solution of a mean field game (MFG) \emph{with common noise} from a solution of a certain MFG \emph{without common noise}. This provides a simple way to extend to the common noise setting many existing results on MFGs without common noise. The MFG with common noise we consider is described concisely as follows:

\begin{problem}{with common noise} \label{mfgcn}
\[
\begin{cases}
&\alpha^* \in \arg\max_\alpha\E\left[\int_0^Tf(t,X^\alpha_t,\mu_t,\alpha_t)dt + g(X^\alpha_T,\mu_T)\right], \\
&dX^\alpha_t = \left[b_0(t,\mu_t) + b(t,X^\alpha_t,\mu_t,\alpha_t)\right]dt \\
	&\quad\quad\quad + \sigma(t,X^\alpha_t,\mu_t,\alpha_t)dW_t + \sigma_0(t,\mu_t)dB_t, \ X^\alpha_0 = \xi, \\
&\mu = \text{Law}(X^{\alpha^*} \ | \ B).
\end{cases}
\]
\end{problem}
Here $\xi$ is some given initial state, $X^\alpha$ is the state process subject to the control $\alpha$, and $\mu$ is a random measure on the path space with time-marginals $(\mu_t)_{t \in [0,T]}$. Definition \ref{def:cnsolution} will formulate this precisely, but for now a more careful explanation is as follows: Given a random measure $\mu$, treat it as \emph{fixed} and solve the stochastic optimal control problem defined in the first two lines above. If an optimal control $\alpha^*$ may be found, compute the conditional law of $X^{\alpha^*}$ given the common noise $B$. If the resulting conditional law matches $\mu$, then we say $\mu$ is an MFG equilibrium. The goal of this paper is to link this problem with the following MFG without common noise:
\begin{problem}{without common noise} \label{mfgncn}
\[
\begin{cases}
&\alpha^* \in \arg\max_\alpha\E\left[\int_0^Tf(t,Y^\alpha_t,\bar{\mu}_t,\alpha_t)dt + g(Y^\alpha_T,\bar{\mu}_T)\right], \\
&dY^\alpha_t = b(t,Y^\alpha_t,\bar{\mu}_t,\alpha_t)dt + \sigma(t,Y^\alpha_t,\bar{\mu}_t,\alpha_t)dW_t, \ Y^\alpha_0 = \xi,  \\
&\bar{\mu} = \text{Law}(Y^{\alpha^*}).
\end{cases}
\]
\end{problem}
The structure of the problem will be made completely precise in Definition \ref{def:ncnsolution}. It is exactly like the common noise problem except that now $\bar{\mu}$ is a \emph{deterministic} measure, matched to the (unconditional) law of the optimally controlled state process $Y^{\alpha^*}$. The crucial structural condition that allows us to relate these two problems is \emph{translation invariance}; we assume that $b$, $\sigma$, $f$, and $g$ satisfy a condition of the form $b(t,x+q,\mu,a) = b(t,x,\mu(\cdot + q),a)$, for all $q$. 
The procedure for constructing common-noise solutions is as follows:
\begin{enumerate}
\item Solve the MFG \emph{without} common noise to get $\bar{\mu}$ and $\alpha^*$.
\item Using $\bar{\mu}$ from step (1), solve the SDE
\[
dq_t = b_0(t,\bar{\mu}_t(\cdot - q_t))dt + \sigma_0(t,\bar{\mu}_t(\cdot - q_t))dB_t, \ q_0 = 0.
\]
\item An equilibrium for the MFG \emph{with} common noise is then given by $\mu := \bar{\mu}(\cdot - q)$, and the same control $\alpha^*$ is optimal. (Define $X^\alpha := Y^\alpha + q$ for all $\alpha$.)
\end{enumerate}
The role of the translation invariance is to isolate the effect of the common noise by decomposing the equilibrium measure flow $(\mu_t)_{t \in [0,T]}$ into a deterministic measure flow $(\bar{\mu}_t)_{t \in [0,T]}$ shifted by a highly tractable finite-dimensional stochastic process $(q_t)_{t \in [0,T]}$. 
We show also that one may invert this procedure to construct a solution of the MFG \emph{without} common noise from a solution of the MFG \emph{with} common noise, if one is willing to work with a weaker notion of solution. While this is less obviously useful, it enables uniqueness arguments and thus completes the connection between the two systems. For the sake of concreteness we work with a finite time horizon $T > 0$, but it should be clear from the analysis that our construction is more broadly applicable, for example to infinite-horizon or ergodic objectives. Finally, it should be remarked that our construction also works on the level of the $n$-player game, but for the sake of brevity we will not discuss this.

While no concrete MFG models are presented in this note, our results provide a tractable method for incorporating common noise terms in MFG models, which often makes the models more realistic or robust. For example, using our construction, the models of population distribution of \cite{gueantlasrylionsmfg} and the flocking models of \cite{nourian-cuckersmalemfg1,carmonalacker-probabilisticweakformulation} can easily be extended to include common noise. The common noise systemic risk model of \cite{carmonafouque-systemicrisk} fits perfectly into our framework, although a direct analysis was possible for this model because of its relatively simple linear-quadratic structure.

Our construction is inspired by the paper \cite{gueantlasrylionsmfg} of Gu\'eant, Lasry, and Lions, in which the equilibrium is computed explicitly for a specific common-noise MFG model of income distribution. While this equilibrium measure flow $(\mu_t)_{t \in [0,T]}$ is indeed random, it may be decomposed into
\[
\mu_t = \nu(q_t \cdot), \text{ i.e. } \mu_t(A) = \nu(q_tA) \text{ for measurable } A \subset \R,
\]
where $\nu$ is a deterministic (Pareto) distribution and $q_t$ is a \emph{one-dimensional} stochastic process. 
This is a multiplicative decomposition, whereas our decompositions are additive. 
Conceivably, many other classes of common-noise MFG models may permit similar decompositions, in which the common noise has a simple finite-dimensional effect on the measure flow.

This note is a contribution to the wellposedness theory for MFGs. The theory was introduced by Lasry and Lions \cite{lasrylionsmfg} and Huang, Malham\'e, and Caines in \cite{huangmfg1,huangmfg3}, largely as a tool for studying limits and approximations of Nash equilibria for corresponding $n$-player games of a certain symmetric type. When $n$ is large, $n$-player stochastic differential games are highly intractable, and the MFG limit is often easier to analyze while still providing a good approximation of the more realistic $n$-player system. This has naturally led to a substantial literature on existence and uniqueness for MFG equilibria; for additional background see the surveys \cite{cardaliaguet-mfgnotes,gomessaude-mfgsurvey} or the more probabilistic \cite{carmonadelarue-mfg,bensoussan-mfgbook}.

To the best of the authors' knowledge, the only general existence results for MFGs with common noise appear in the two recent papers \cite{carmonadelaruelacker-mfgcommonnoise,ahuja-mfgwellposedness}, although common noise had appeared already in specific models in \cite{gueantlasrylionsmfg,carmonafouque-systemicrisk}. Very recently, the papers \cite{carmonadelarue-master,bensoussanfrehseyam-master} derive the so-called \emph{master equation}, which reformulates the problem in terms of a single PDE, but  no existence results are provided. Under a monotonicity assumption similar to that of Lasry and Lions \cite{lasrylionsmfg}, Ahuja \cite{ahuja-mfgwellposedness} proves existence and uniqueness for a class of nearly linear-quadratic MFGs with common noise. The existence results of \cite{carmonadelaruelacker-mfgcommonnoise} apply to much more general systems but provide only \emph{weak solutions}, the main differences being that $\mu$ is not necessarily $B$-measurable and that the fixed point condition $\mu = \text{Law}(X^{\alpha^*} \ | \ B)$ is replaced by the weaker condition $\mu = \text{Law}(X^{\alpha^*} \ | \ B,\mu)$. The present paper provides new results on \emph{strong} solutions, but only for the particular class of \emph{translation invariant} MFG models. However, the objective of this note is not so much to prove a precise existence result but rather to provide a mechanism for constructing strong solutions to a large class of common noise mean field games, in a surprisingly tractable manner.

The first main result, Theorem \ref{th:main}, shows how to construct a strong solution of the problem with common noise from a strong solution of the problem without common noise. 
This construction generalizes to weak solutions as well, and our second main result, Theorem \ref{th:mainconverse}, is a converse, allowing us to construct a weak solution of the problem without common noise from a weak solution of the problem with common noise. 
In \cite{lacker-meanfieldlimit}, it is proven that the set of \emph{weak solutions} introduced in \cite{carmonadelaruelacker-mfgcommonnoise} precisely characterizes the set of possible limits of approximate Nash equilibria of the corresponding $n$-player games, as the number of agents $n$ tends to infinity. (The sense in which this limit is meant is made clear in \cite{lacker-meanfieldlimit}.) Hence, Theorem \ref{th:mainconverse} (or more specifically its Corollary \ref{co:uniqueness}) is useful from the perspective of the $n$-player games, as it shows there is a complete correspondence between the sets of weak solutions of the two problems, with and without common noise.

The main result is set up in Section \ref{se:main} stated precisely in Theorem \ref{th:main}. Section \ref{se:application} presents some applications, including a precise existence result with easily verifiable assumptions. Section \ref{se:uniqueness} discusses the converse to Theorem \ref{th:main}, completing the connection between the two MFG problems.

\section{Main results} \label{se:main}
Given a Polish space $E$, let $\P(E)$ denote the Borel probability measures on $E$. Endow $\P(E)$ with the topology of weak convergence and the corresponding Borel $\sigma$-field. Let $\C^k := C([0,T];\R^k)$ denote the space of continuous $\R^k$-valued paths, endowed with the supremum norm and Borel $\sigma$-field. Given $\mu \in \P(\C^k)$ and $t \in [0,T]$, let $\mu_t \in \P(\R^k)$ denote the image of the projection $x \mapsto x_t$ under $\mu$. Given $\mu \in \P(\R^k)$ (resp. $\P(\C^k)$) and $q \in \R^k$ (resp. $\C^k$), the translation by $-q$ of $\mu$ is denoted $\mu(\cdot + q)$, which is the image of map $x \mapsto x - q$ under $\mu$. The key assumption will involve the following definition:

\begin{definition}
A subset $D$ of $\P(\R^d)$ is said to be \emph{translation invariant} if $\mu(\cdot + q) \in D$ for each $q \in \R^k$ and $\mu \in D$. Given such a $D$ and a function $F : \R^d \times D \rightarrow E$ for some set $E$, we say $F$ is \emph{translation invariant} if, for each $x,q \in \R^k$ and $\mu \in D$,
\[
F(x+q,\mu) = F(x,\mu(\cdot + q)),
\]
\end{definition}

The domains $D$ of interest to us are the entire space $\P(\R^k)$ and the subset consisting of measures admitting Lebesgue-densities. The guiding examples of translation invariant functions $F$ are convolutions,
\[
F(x,\mu) = G\left(\int\phi(x-y)\mu(dy)\right),
\]
and local interactions (noting that $D = \{\mu \in \P(\R^k) : \mu \ll \text{Lebesgue}\}$ is translation invariant),
\[
F(x,\mu) = G\left(\frac{d\mu}{dx}(x)\right).
\]
The term \emph{translation invariant} is chosen because of the equivalent definition that $F(x+q,\mu(\cdot - q)) = F(x,\mu)$, which shows that $F$ is unchanged when the same translation is applies to both the spatial variable and the measure.

We are given the following data. The control space $A$ is a closed subset of Euclidean space, $\lambda \in \P(\R^d)$ is an initial state distribution, and $d$, $m$, and $m_0$ are positive integers. We are given an exponent $p \ge 0$, simply to specify some class of admissible controls by way of an integrability assumption. A translation invariant domain $D \subset \P(\R^d)$ and the following functions are given:
\begin{align*}
(b,\sigma,f) &: [0,T] \times \R^d \times D \times A \rightarrow \R^d \times \R^{d \times m} \times \R, \\
(b_0,\sigma_0) &: [0,T] \times D \rightarrow \R^d \times \R^{d \times m_0}, \\
g &: \R^d \times D \rightarrow \R.
\end{align*}

\begin{assumption}
Each function is jointly measurable, and for each fixed $(t,a) \in [0,T] \times A$ the functions $b$, $\sigma$, $f$, and $g$ are translation invariant on $\R^d \times D$.
\end{assumption}

The assumptions on the coefficients are minimal for now, but condition (3) of each of the following definitions will implicitly require that certain integrals make sense. Consider the following notions of MFG solution:

\begin{definition}[Strong common-noise solution] \label{def:cnsolution}
A \emph{strong solution of $CN(b,\sigma,f,g,b_0,\sigma_0)$} is a tuple $(\Omega,(\F_t)_{t \in [0,T]},P,B,W,\mu,\alpha,X)$, where $(\Omega,(\F_t)_{t \in [0,T]},P)$ is a complete filtered probability space supporting $(B,W,\mu,\alpha,X)$ satisfying the following:
\begin{enumerate}
\item $B$ and $W$ are independent $(\F_t)_{t \in [0,T]}$-Wiener processes of dimension $m_0$ and $m$, respectively, and $X$ is a continuous $d$-dimensional $(\F_t)_{t \in [0,T]}$-adapted process with $P \circ X_0^{-1} = \lambda$.
\item $\alpha$ is an $(\F_t)_{t \in [0,T]}$-progressive $A$-valued process satisfying $\E^P\int_0^T|\alpha_t|^pdt < \infty$.
\item The state equation holds:
\begin{align}
dX_t = &[b_0(t,\mu_t) + b(t,X_t,\mu_t,\alpha_t)]dt + \sigma(t,X_t,\mu_t,\alpha_t)dW_t + \sigma_0(t,\mu_t)dB_t. \label{def:SDE-X}
\end{align}
\item If $(\Omega',\F'_t,P')$ is another filtered probability space supporting processes $(B',W',\mu',\alpha',X')$ satisfying (1-3) and $P \circ (B,\mu)^{-1} = P' \circ (B',\mu')^{-1}$, then
\[
\E^P\left[\int_0^Tf(t,X_t,\mu_t,\alpha_t)dt + g(X_T,\mu_T)\right] \ge \E^{P'}\left[\int_0^Tf(t,X'_t,\mu'_t,\alpha'_t)dt + g(X'_T,\mu'_T)\right].
\]
\item $\mu$ is a random element of $\P(\C^d)$ such that $\mu = P(X \in \cdot \ | \ B)$ a.s., and $\mu_t \in D$ a.s. for each $t \in [0,T]$.
\end{enumerate}
\end{definition}

\begin{remark}
Note that $\mu = P(X \in \cdot \ | \ B)$ implies that $\mu_t = P(X_t \in \cdot \ | \ B) = P(X_t \in \cdot \ | \ \F^B_t)$, where $\F^B_t = \sigma(B_s : s \le t)$, since $(X_s,B_s)_{s \in [0,t]}$ and $(B_s - B_t)_{s \in [t,T]}$ are independent. That is, $(\mu_t)_{t \in [0,T]}$ is $(\F^B_t)_{t \in [0,T]}$-adapted.
\end{remark}

\begin{definition}[Strong no-common-noise solution] \label{def:ncnsolution}
A \emph{strong solution of $NCN(b,\sigma,f,g)$} is a tuple $(\Omega,(\F_t)_{t \in [0,T]},P,W,\bar{\mu},\alpha,Y)$, where $(\Omega,(\F_t)_{t \in [0,T]},P)$ is a complete filtered probability space supporting $(W,\alpha,Y)$ satisfying the following:
\begin{enumerate}
\item $W$ is a $(\F_t)_{t \in [0,T]}$-Wiener processes of dimension $m$, respectively, and $Y$ is a continuous $d$-dimensional $(\F_t)_{t \in [0,T]}$-adapted process with $P \circ Y_0^{-1} = \lambda$.
\item $\alpha$ is an $(\F_t)_{t \in [0,T]}$-progressive $A$-valued process satisfying $\E^P\int_0^T|\alpha_t|^pdt < \infty$.
\item The state equation holds:
\begin{align}
dY_t &= b(t,Y_t,\bar{\mu}_t,\alpha_t)dt + \sigma(t,Y_t,\bar{\mu}_t,\alpha_t)dW_t. \label{def:SDE-Y}
\end{align}
\item If $(\Omega',\F'_t,P')$ is another filtered probability space supporting processes $(B',W',\alpha',Y')$ satisfying (1-3), then
\[
\E^P\left[\int_0^Tf(t,Y_t,\bar{\mu}_t,\alpha_t)dt + g(Y_T,\bar{\mu}_T)\right] \ge \E^{P'}\left[\int_0^Tf(t,Y'_t,\bar{\mu}_t,\alpha'_t)dt + g(Y'_T,\bar{\mu}_T)\right].
\]
\item $\bar{\mu} \in \P(\C^d)$ satisfies $\bar{\mu} = P \circ Y^{-1}$ and $\bar{\mu}_t \in D$ for all $t \in [0,T]$.
\end{enumerate}
\end{definition}

The first main result of the paper is the following:

\begin{theorem} \label{th:main}
Suppose $(\Omega,(\F_t)_{t \in [0,T]},P,W,\bar{\mu},\alpha,Y)$ is a strong solution of $NCN(b,\sigma,f,g)$. By extending the probability space, we may assume that $\Omega$ supports an $m_0$-dimensional $(\F_t)_{t \in [0,T]}$-Wiener process $B$ independent of $(W,\alpha,Y)$. Suppose weak existence and pathwise uniqueness hold for the following SDE:
\begin{align}
dq_t = b_0(t,\bar{\mu}_t(\cdot - q_t))dt + \sigma_0(t,\bar{\mu}_t(\cdot - q_t))dB_t, \ q_0 = 0. \label{def:qSDE}
\end{align}
If $X := Y + q$ and $\mu := \bar{\mu}(\cdot - q)$, then $(\Omega,(\F_t)_{t \in [0,T]},P,B,W,\mu,\alpha,X)$ is a strong solution of $CN(b,\sigma,f,g,b_0,\sigma_0)$.
\end{theorem}
\begin{proof}
Solve the SDE \eqref{def:qSDE} on $\Omega$, and note that $(q_t)_{t \in [0,T]}$ is adapted to (the $P$-completion of) $(\F^B_t := \sigma(B_s : s \le t))_{t \in [0,T]}$ and that $\mu$ is the translation of the \emph{deterministic} measure $\bar{\mu} \in \P(\C^d)$ by the \emph{stochastic} process $q$.
First, note that translation invariance implies $b(t,Y_t,\bar{\mu}_t,\alpha_t) = b(t,X_t,\mu_t,\alpha_t)$, and similarly for $\sigma$, $f$, and $g$. Thus
\begin{align*}
dX_t &= dY_t + dq_t \\
	&= [b_0(t,\bar{\mu}_t(\cdot - q_t)) + b(t,Y_t,\bar{\mu}_t,\alpha_t)]dt + \sigma(t,Y_t,\bar{\mu}_t,\alpha_t)dW_t + \sigma_0(t,\bar{\mu}_t(\cdot - q_t))dB_t, \\
	&= [b_0(t,\mu_t) + b(t,X_t,\mu_t,\alpha_t)dt] + \sigma(t,X_t,\mu_t,\alpha_t)dW_t + \sigma_0(t,\mu_t)dB_t.
\end{align*}
Note also that $X_0 = Y_0$. Since $\bar{\mu} = P \circ Y^{-1}$, since $Y$ is independent of $B$, and since $q$ is $B$-measurable, we have
\begin{align*}
\mu &= \bar{\mu}(\cdot - q) = P( Y \in \cdot - q \ | \ B ) = P( X \in \cdot \ | \ B).
\end{align*}
It remains to check the optimality property (4) of Definition \ref{def:cnsolution}. Let $(\Omega',(\F'_t)_{t \in [0,T]},P')$ be any filtered probability space supporting processes $(B',W',\mu',\alpha',X')$ satisfying (1-3) of Definition \ref{def:cnsolution} and $P' \circ (B',\mu')^{-1} = P \circ (B,\mu)^{-1}$. 
Let $(q'_t)_{t \in [0,T]}$ denote the unique strong solution on $\Omega'$ of the SDE
\[
dq'_t = b_0(t,\bar{\mu}_t(\cdot - q'_t))dt + \sigma_0(t,\bar{\mu}_t(\cdot - q'_t))dB'_t, \ q'_0 = 0.
\]
Then $q'$ is $B'$-measurable and $P' \circ (B',q')^{-1} = P \circ (B,q)^{-1}$ by uniqueness. Thus $P' \circ (B',\mu',q')^{-1} = P \circ (B,\mu,q)^{-1}$, since $\mu'$ (resp. $\mu$) is $B'$-measurable (resp. $B$-measurable), and we conclude that $\mu' = \bar{\mu}(\cdot - q')$ a.s. Define $Y' := X' - q'$, and again use translation invariance of $b$ and $\sigma$ to get
\begin{align*}
dY'_t &= dX'_t - dq'_t = b(t,X'_t,\mu'_t,\alpha'_t)dt + \sigma(t,X'_t,\mu'_t,\alpha'_t)dW_t, \\
	&= b(t,Y'_t,\bar{\mu}_t,\alpha'_t)dt + \sigma(t,Y_t,\bar{\mu}_t,\alpha'_t)dW_t.
\end{align*}
Since $(\Omega,(\F_t)_{t \in [0,T]},P,W,\alpha,Y)$ is a NCN solution, we may apply the optimality condition (4) of Definition \ref{def:ncnsolution} and then translation invariance to get
\begin{align*}
0 &\le \E^{P}\left[\int_0^Tf(t,Y_t,\bar{\mu}_t,\alpha_t)dt + g(Y_T,\bar{\mu}_T)\right] - \E^{P'}\left[\int_0^Tf(t,Y'_t,\bar{\mu}_t,\alpha'_t)dt + g(Y'_T,\bar{\mu}_T)\right] \\
	&= \E^{P}\left[\int_0^Tf(t,X_t,\mu_t,\alpha_t)dt + g(X_T,\mu_T)\right] - \E^{P'}\left[\int_0^Tf(t,X'_t,\mu'_t,\alpha'_t)dt + g(X'_T,\mu'_T)\right].
\end{align*}
\end{proof}

Note that the definition of NCN solution involves \emph{weak controls}, according to the terminology of \cite{carmonadelaruelacker-mfgcommonnoise}, which are not required to be adapted to the filtration $\F_t = \sigma(Y_0,W_s : s \le t)$ generated by the given sources of randomness. Only the recent papers \cite{lacker-mfgviacontrolledmtgproblems,fischer-mfgconnection,carmonalacker-probabilisticweakformulation} work with essentially the same Definition \ref{def:ncnsolution}, using weak controls; most probabilistic notions of MFG solutions in the literature restrict their attention to strong controls, adapted to $\sigma(Y_0,W_s : s \le t)$. But it is rather well-known from the theory of relaxed controls that this additional flexibility does not typically help the agent; a control which is optimal among the class of strong controls is typically also optimal among the class of weak controls. Thus, a special case of our notion of strong NCN solution is the usual MFG solution appearing in the probabilistic literature (e.g. \cite{carmonadelarue-mfg,bensoussan-mfgbook}). MFG solutions obtained by PDE methods are translated to the stochastic setting through verification theorems which allow for weak controls, and thus our notion of solution includes those defined by (classical) PDE solutions (e.g. \cite{lasrylionsmfg}). 
To be somewhat more precise:

\begin{proposition}
Suppose that $b$ and $\sigma$ are uniformly Lipschitz in $x$, $A$ is compact, and $b$, $\sigma$, $f$, and $g$ are all jointly continuous in $(x,a)$. Then it is equivalent in Definition \ref{def:cnsolution} to replace condition (4) with one requiring that $\alpha'$ be adapted to (the $P'$-completion of) $\sigma(X'_0,B'_s,W'_s : s \le t)$. Similarly, in Definition \ref{def:ncnsolution}, it is equivalent to require optimality only among $\sigma(Y'_0,W'_s:s\le t)$-adapted controls.
\end{proposition}
\begin{proof}
See, for example, \cite{elkaroui-compactification} or the recent account of \cite{elkarouitan-capacities}.
\end{proof}

\begin{remark}
The construction of CN solutions in Theorem \ref{th:main} leaves the optimal control unchanged, and thus we may construct a CN solution whose optimal control is independent of the common noise $B$ and, a fortiori, independent of the random measure $\mu$. Intuitively, the entire population is affected in parallel by the common noise through $(q_t)_{t \in [0,T]}$, and because of the translation invariance the common noise does not influence the optimization.
\end{remark}

\section{Applications} \label{se:application}
This section discusses examples of applications of Theorem \ref{th:main}. First, some comments on the SDE \eqref{def:qSDE} are in order. Given $p \ge 1$, define $\P^p(\R^d)$ to be the set of $\mu \in \P(\R^d)$ with $\int|x|^p\mu(dx) < \infty$. Define the $p$-Wasserstein distance $\W_p$ on $\P^p(\R^d)$ by
\[
\W_p(\mu,\nu) := \inf_\gamma\left(\int_{\R^d \times \R^d}|x-y|^p\gamma(dx,dy)\right)^{1/p},
\]
where the infimum is over $\gamma \in \P(\R^d \times \R^d)$ with marginals equal to $\mu$ and $\nu$.
The assumption of solvability of the SDE \eqref{def:qSDE} is guaranteed by assuming that the coefficients $b_0(t,\mu)$ and $\sigma_0(t,\mu)$ are $\W_p$-Lipschitz in $\mu$, uniformly with respect to $t$. Indeed, for any $\mu \in \P^p(\R^d)$ and $q \in \R$ we have
\[
\W_p(\mu(\cdot - q), \mu(\cdot - q')) \le \left(\int_{\R^d}|(x+q) - (x+q')|^p\mu(dx)\right)^{1/p} = |q-q'|,
\]
and it follows that $b_0(t,\bar{\mu}_t(\cdot - q))$ and $\sigma_0(t,\bar{\mu}_t(\cdot - q))$ are Lipschitz in $q$, uniformly with respect to $t$, for each $\bar{\mu} \in \P^p(\C^d)$.

Theorem \ref{th:main} allows us to derive common noise existence results from the existence results without common noise of \cite{lacker-mfgviacontrolledmtgproblems}. Take $D = \P^p(\R^d)$ in the following.

\begin{theorem} \label{th:application}
Under the following assumptions, there exists a strong solution of $CN(b,\sigma,f,g,b_0,\sigma_0)$:
\begin{enumerate}
\item The control space $A$ is a closed subset of a Euclidean space.
\item The initial distribution $\lambda$ is in $\P^{p'}(\R^d)$, where $p' > p \ge \max\{1, p_\sigma\}$, $p_\sigma \in [0,2]$.
\item The functions $b$, $b_0$, $\sigma$, $\sigma_0$, $f$, and $g$ of $(t,x,\mu,a)$ are measurable in $t$ and continuous in $(x,\mu,a)$ (with respect to the metric $\W_p$ on $\P^p(\R^d)$).
\item There exists $c_1 > 0$ such that, for all $(t,\mu,a) \in [0,T] \times \P^p(\R^d) \times A$ and all $x,y \in \R^d$,
\begin{align*}
|b(t,x,\mu,a) - b(t,y,\mu,a)| + |\sigma(t,x,\mu,a) - \sigma(t,y,\mu,a)| &\le c_1|x-y|,
\end{align*}
and
\begin{align*}
|b(t,x,\mu,a)| &\le c_1\left[1 + |x| + \left(\int_{\R^d}|z|^p\mu(dz)\right)^{1/p} + |a|\right], \\
|\sigma\sigma^\top(t,x,\mu,a)| &\le c_1\left[1 + |x|^{p_\sigma} + \left(\int_{\R^d}|z|^p\mu(dz)\right)^{p_\sigma/p} + |a|^{p_\sigma}\right]
\end{align*}
\item There exist $c_2, c_3 > 0$ such that, for each $(t,x,\mu,a) \in [0,T] \times \R^d \times \P^p(\R^d) \times A$, 
\begin{align*}
|g(x,\mu)| &\le c_2\left(1 + |x|^p + |\mu|^p\right), \\
-c_2\left(1 + |x|^{p} + |\mu|^p + |a|^{p'}\right) \le f(t,x,\mu,a) &\le c_2\left(1 + |x|^p + |\mu|^p\right) - c_3|a|^{p'}
\end{align*}
\item For each $(t,x,\mu)$ the following subset of $\R^d \times \R^{d \times d} \times \R$ is convex:
\[
\left\{(b(t,x,\mu,a),\sigma\sigma^\top(t,x,\mu,a),z) : a \in A, \ z \le f(t,x,\mu,a)\right\}
\]
\item The functions $b$, $\sigma$, $f$, and $g$ are translation invariant in $(x,\mu)$, for each $(t,a)$.
\item There exists $c_4 > 0$ such that, for each $t \in [0,T]$ and $\mu,\nu \in \P^p(\R^d)$,
\[
|b_0(t,\mu) - b_0(t,\nu)| + |\sigma_0(t,\mu) - \sigma_0(t,\nu)| \le c_4\W_p(\mu,\nu).
\]
\end{enumerate}
\end{theorem}
\begin{proof}
Assumptions (1-6) and \cite[Theorem 2.1]{lacker-mfgviacontrolledmtgproblems} imply that there exists an NCN solution. By assumption (8), as pointed out before, the SDE \eqref{def:qSDE} is well-posed. In light of assumption (7), Theorem \ref{th:main} applies. 
\end{proof}

Similarly, combining our Theorem \ref{th:main} with any of the results on MFG without common noise with \emph{local interactions}, such as \cite{gomes-mfgsubquadratic,gomes-mfgsuperquadratic,cardaliaguet-mfgdegenerate,lasrylionsmfg,cardaliaguet-mfgnotes,gomessaude-mfgsurvey}, we could derive existence results for some mean field games with common noise and local interactions. Aside from the specific model of \cite{gueantlasrylionsmfg}, local interactions have not been incorporated in common noise models before, and without our construction this would presumably be quite a technical matter. As it should be clear at this stage how to construct such results, and since spelling them out in detail would require yet another laundry list of technical assumptions, we suppress any further details.

We conclude the section with a useful trick which allows us to apply our theorem to certain \emph{nearly} translation invariant functions. 
Given coefficients $(b,\sigma,f,g)$ satisfying the standing assumptions, define new coefficients
\begin{align*}
\tilde{b}(t,x,\mu,a) &= Qx + b(t,x,\mu,a), \\
\tilde{f}(t,x,\mu,a) &= r_f \cdot x + f(t,x,\mu,a), \\
\tilde{g}(x,\mu) &= r_g \cdot x + g(x,\mu).
\end{align*}
where $Q$ is a $d \times d$ matrix and $r_f,r_g \in \R^d$. Define also
\[
\tilde{b}_0(\mu) := Q\int_{\R^d}y\,\mu(dy), \quad \tilde{f}_0(\mu) := \int_{\R^d}r_f \cdot y\,\mu(dy), \quad \tilde{g}_0(\mu) := \int_{\R^d}r_g \cdot y\,\mu(dy).
\]
Naively, Theorem \ref{th:main} does not apply to the coefficients $(\tilde{b},\sigma,\tilde{f},\tilde{g},b_0,\sigma_0)$, since $\tilde{b}$, $\tilde{f}$, and $\tilde{g}$ are not translation invariant.
However, simply obvserve that every solution of $CN(\tilde{b},\sigma,\tilde{f},\tilde{g},b_0,\sigma_0)$ is also a solution of $CN(\tilde{b}-\tilde{b}_0,\sigma,\tilde{f} - \tilde{f}_0,\tilde{g} - \tilde{g}_0,b_0+\tilde{b}_0,\sigma_0)$, and the converse is true as well. Note, of course, that subtracting $\tilde{f}_0$ and $\tilde{g}_0$ does not alter the optimization problems. The point is that the latter coefficients are translation invariant, and thus we may use Theorem 2.6 to construct a solution of these CN problems from a solution of $NCN(\tilde{b}-\tilde{b}_0,\sigma,\tilde{f} - \tilde{f}_0,\tilde{g} - \tilde{g}_0)$. For example, this trick would allow us to effortlessly incorporate common noise into the flocking models considered in \cite{nourian-cuckersmalemfg1,carmonalacker-probabilisticweakformulation}, which have a linear drift term which is not translation invariant.

\section{Weak solutions and the converse} \label{se:uniqueness}

To make sense of the converse of Theorem \ref{th:main}, we will need the notion of \emph{weak solution} for mean field games with and without common noise, introduced in \cite{carmonadelaruelacker-mfgcommonnoise} and \cite{lacker-meanfieldlimit}, respectively. The meaning of \emph{weak} here is probabilistic: the random measure $\mu$ in the CN solution is no longer required to be $B$-measurable, and the measure $\bar{\mu}$ of the NCN solution is now allowed to be random. Unfortunately, there is an additional subtely in the definitions which necessitates some more notation. Let $\X := \C^m \times L^p([0,T];A) \times \C^d$, and let $(\F^\X_t)_{t \in [0,T]}$ denote the natural filtration on $\X$, where $\F^\X_t$ is the $\sigma$-field generated by the maps
\begin{align*}
\X \ni (w,\alpha,x) &\mapsto (w_s,x_s) \in \R^m \times \R^d, \text{ for } s \le t, \text{ and } \\
\X \ni (w,\alpha,x) &\mapsto \int_0^s1_C(\alpha_s)ds, \text{ for } s \le t, \ C \in \B(A). 
\end{align*}
A measure $\mu \in \P(\X)$ is to represent a joint law of $(W,\alpha,X)$: the independent noise, the control, and the state process. Given $\mu \in \P(\X)$, let $\mu^x := \mu(\C^m \times \L^p([0,T];A) \times \cdot)$ denote the $\C^d$-marginal. Given $q \in \C^d$, we may write $\mu(\cdot + (0,0,q))$ to denote the translation of $\mu$ in the direction of $\C^d$ given by $q$, defined by
\[
\mu(A + (0,0,q)) := \mu\left\{(w,\alpha,x + q) : (w,\alpha,x) \in A\right\},
\]
Note of course that $\mu(\cdot + (0,0,q))^x = \mu^x(\cdot + q)$.
The following two definitions come from \cite{carmonadelaruelacker-mfgcommonnoise} and \cite{lacker-meanfieldlimit}, respectively. The curious reader is referred to these papers for a thorough discussion of these definitions, especially \cite{carmonadelaruelacker-mfgcommonnoise} for the unusual conditional independence of condition (3) and the necessity of considering measures on the larger space $\X$, rather than $\C^d$.

\begin{definition}[Weak common-noise solution] \label{def:wcnsolution}
A \emph{weak solution of $CN(b,\sigma,f,g,b_0,\sigma_0)$} is a tuple $(\Omega,(\F_t)_{t \in [0,T]},P,B,W,\mu,\alpha,X)$, where $(\Omega,(\F_t)_{t \in [0,T]},P)$ is a complete filtered probability space supporting $(B,W,\mu,\alpha,X)$ satisfying the following:
\begin{enumerate}
\item $B$ and $W$ are independent $(\F_t)_{t \in [0,T]}$-Wiener processes of dimension $m_0$ and $m$, respectively, and $X$ is a continuous $d$-dimensional $(\F_t)_{t \in [0,T]}$-adapted process with $P \circ X_0^{-1} = \lambda$.
\item $\mu$ is a random element of $\P(\X)$, and $\mu(C)$ is $\F_t$-measurable for each $C \in \F^\X_t$ and $t \in [0,T]$.
\item $\alpha$ is a $(\F_t)_{t \in [0,T]}$-progressively measurable $A$-valued process satisfying $\E^P\int_0^T|\alpha_t|^pdt < \infty$, and $\sigma(\alpha_s : s \le t)$ is conditionally independent of $\F^{X_0,B,W,\mu}_T$ given $\F^{X_0,B,W,\mu}_t$, for each $t \in [0,T]$, where
\[
\F^{X_0,B,W,\mu}_t := \sigma(X_0,B_s,W_s,\mu(C) : s \le t, \ C \in \F^\X_t).
\]
\item $X_0$, $W$, and $(B,\mu)$ are independent.
\item The state equation \eqref{def:SDE-X} holds.
\item If $(\Omega',(\F'_t)_{t \in [0,T]},P')$ is another filtered probability space supporting $(B',W',\mu',\alpha',X')$ satisfying (1-5) and $P \circ (B,\mu)^{-1} = P' \circ (B',\mu')^{-1}$, then
\[
\E^P\left[\int_0^Tf(t,X_t,\mu^x_t,\alpha_t)dt + g(X_T,\mu^x_T)\right] \ge \E^{P'}\left[\int_0^Tf(t,X'_t,\mu'^x_t,\alpha'_t)dt + g(X'_T,\mu'^x_T)\right].
\]
\item $\mu = P((W,\alpha,X) \in \cdot \ | \ (B,\mu))$ a.s., and $\mu^x_t \in D$ a.s. for each $t \in [0,T]$.
\end{enumerate}
\end{definition}

\begin{definition}[Weak no-common-noise solution] \label{def:wncnsolution}
A \emph{weak solution of $NCN(b,\sigma,f,g)$} is a tuple $(\Omega,(\F_t)_{t \in [0,T]},P,W,\bar{\mu},\alpha,Y)$, where $(\Omega,(\F_t)_{t \in [0,T]},P)$ is a filtered probability space supporting $(W,\bar{\mu},\alpha,Y)$ satisfying the following:
\begin{enumerate}
\item $W$ is a $(\F_t)_{t \in [0,T]}$-Wiener processes of dimension $m$, and $Y$ is a continuous $d$-dimensional $(\F_t)_{t \in [0,T]}$-adapted process with $P \circ Y_0^{-1} = \lambda$.
\item $\bar{\mu}$ is a random element of $\P(\X)$, and $\bar{\mu}(C)$ is $\F_t$-measurable for each $C \in \F^\X_t$ and $t \in [0,T]$.
\item $\alpha$ is an $(\F_t)_{t \in [0,T]}$-progressively measurable $A$-valued process satisfying $\E^P\int_0^T|\alpha_t|^pdt < \infty$, and $\sigma(\alpha_s : s \le t)$ is conditionally independent of $\F^{X_0,W,\bar{\mu}}_T$ given $\F^{X_0,W,\bar{\mu}}_t$, for each $t \in [0,T]$, where
\[
\F^{X_0,W,\bar{\mu}}_t := \sigma(X_0,W_s,\bar{\mu}(C) : s \le t, \ C \in \F^\X_t).
\]
\item $X_0$, $W$, and $\bar{\mu}$ are independent.
\item The state equation \eqref{def:SDE-Y} holds.
\item If $(\Omega',(\F'_t)_{t \in [0,T]},P')$ is another filtered probability space supporting $(W',\bar{\mu}',\alpha',Y')$ satisfying (1-5) and $P \circ \bar{\mu}^{-1} = P' \circ (\bar{\mu}')^{-1}$, then
\[
\E^P\left[\int_0^Tf(t,Y_t,\bar{\mu}^x_t,\alpha_t)dt + g(Y_T,\bar{\mu}^x_T)\right] \ge \E^{P'}\left[\int_0^Tf(t,Y'_t,\bar{\mu}'^x_t,\alpha'_t)dt + g(Y'_T,\bar{\mu}'^x_T)\right].
\]
\item $\bar{\mu}= P((W,\alpha,Y) \in \cdot \ | \ \bar{\mu})$ a.s., and $\bar{\mu}^x_t \in D$ a.s. for each $t \in [0,T]$.
\end{enumerate}
\end{definition}

Note that both of these solution notions correspond to \emph{weak MFG solutions with strict control} in \cite{carmonadelaruelacker-mfgcommonnoise}. If $A$ is replaced by $\P(A)$, then these notions correspond to \emph{weak MFG solutions with weak (relaxed) control} in \cite{carmonadelaruelacker-mfgcommonnoise}. Note that a weak CN solution for which $\mu$ happens to be $B$-measurable automatically provides a (strong) CN solution, replacing $\mu$ with $\mu^x$. Similarly, a weak NCN solution for which $\bar{\mu}$ happens to be deterministic (a.s. constant) automatically yields a (strong) NCN solution. One additional definition will facilitate the statement of the final results:

\begin{definition} \label{qSDE-defn}
Consider two filtered probability spaces $(\Omega^i,(\F^i_t)_{t \in [0,T]},P^i)$ supporting (respectively) an $\R^d$-valued adapted process $(q^i_t)_{t \in [0,T]}$, a Wiener process $(B^i_t)_{t \in [0,T]}$, and a $\P(\X)$-valued random variable $\bar{\mu}$ such that $\bar{\mu}^x_t \in D$ a.s. for each $t \in [0,T]$ and such that $\bar{\mu}(C)$ is $\F^i_t$-measurable for each $C \in \F^\X_t$ and each $t \in [0,T]$. Suppose also that 
\begin{align}
dq^i_t = b_0(t,\bar{\mu}^{ix}_t(\cdot - q^x_t))dt + \sigma_0(t,\bar{\mu}^{ix}_t(\cdot - q^i_t))dB^i_t, \ q^i_0 = 0. \label{def:qSDE-random}
\end{align}
Suppose that for any such a pair of spaces satisfying $P^1 \circ (B^1,\bar{\mu}^1)^{-1} = P^2 \circ (B^2,\bar{\mu}^2)^{-1}$ we also have $P^1 \circ (B^1,\bar{\mu}^1,q^1)^{-1} = P^2 \circ (B^2,\bar{\mu}^2,q^2)^{-1}$. Then we say \emph{weak uniqueness holds for the SDE \eqref{def:qSDE-random}}.
\end{definition}

As argued at the beginning of Section \ref{se:application}, weak uniqueness holds for the SDE \eqref{def:qSDE-random} when $b_0$ and $\sigma_0$ are $\W_p$-Lipschitz in the measure argument uniformly in $t$, for some $p \ge 1$.

\begin{theorem} \label{th:mainconverse}
Suppose $(\Omega,(\F_t)_{t \in [0,T]},P,B,W,\mu,\alpha,X)$ is a weak solution of $CN(b,\sigma,f,g,b_0,\sigma_0)$. Suppose that weak uniqueness holds for the SDE \eqref{def:qSDE-random}. Define a process $(q_t)_{t \in [0,T]}$ on $\Omega$ by
\[
q_t = \int_0^tb_0(s,\mu_s)ds + \int_0^t\sigma_0(s,\mu_s)dB_s.
\] 
If $Y := X - q$ and $\bar{\mu} := \mu(\cdot + (0,0,q))$, then $(\Omega,(\F_t)_{t \in [0,T]},P,W,\bar{\mu},\alpha,Y)$ is a weak solution of $NCN(b,\sigma,f,g)$. 
\end{theorem}
\begin{proof}
We simply invert the construction of Theorem \ref{th:main}. Since $\mu = P((W,\alpha,X) \in \cdot \ | \ B,\mu)$ and $q$ is $(B,\mu)$-measurable, 
\begin{align*}
\bar{\mu} &= P((W,\alpha,X) \in \cdot + (0,0,q) \ | \ B,\mu) = P((W,\alpha,Y) \in \cdot \ | \ B,\mu).
\end{align*}
Clearly $\bar{\mu}$ generates a smaller $\sigma$-field than $(B,\mu)$, and so we may condition both sides of this equation on $\bar{\mu}$ to get
\begin{align*}
\bar{\mu} &= P((W,\alpha,Y) \in \cdot \ | \ \bar{\mu}).
\end{align*}
Translation invariance implies $b(t,X_t,\mu^x_t,\alpha_t) = b(t,Y_t,\bar{\mu}^x_t,\alpha_t)$, and similarly for $\sigma$ and $f$. Thus
\begin{align*}
dY_t &= dX_t - dq_t = b(t,X_t,\mu^x_t,\alpha_t)dt + \sigma(t,X_t,\mu^x_t,\alpha_t)dW_t \\
	&= b(t,Y_t,\bar{\mu}^x_t,\alpha_t)dt + \sigma(t,Y_t,\bar{\mu}^x_t,\alpha_t)dW_t.
\end{align*}
Note that $q$ happens to verify the SDE
\[
dq_t = b_0(t,\bar{\mu}^x_t(\cdot - q_t))dt + \sigma_0(t,\bar{\mu}^x_t(\cdot - q_t))dB_t, \ q_0=0.
\]

Now suppose $(\Omega',(\F'_t)_{t \in [0,T]},P')$ is another filtered probability space supporting $(W',\bar{\mu}',\alpha',Y')$ satisfying (1-5) of Definition \ref{def:wncnsolution} and $P' \circ (\bar{\mu}')^{-1} = P \circ \bar{\mu}^{-1}$. By enlarging the space $\Omega'$, we may assume without loss of generality that it supports a $(\F'_t)_{t \in [0,T]}$-Wiener process $B'$ of dimension $m_0$ such that $Y'_0$, $W'$ and $(B',\bar{\mu}')$ are independent and such that $P' \circ (B',\bar{\mu}')^{-1} = P \circ (B,\bar{\mu})^{-1}$. Solve the SDE
\[
dq'_t = b_0(t,\bar{\mu}'^x_t(\cdot - q'_t))dt + \sigma_0(t,\bar{\mu}'^x_t(\cdot - q'_t))dB'_t, \ q'_0=0.
\]
Define $\mu' := \bar{\mu}'(\cdot - (0,0,q'))$. 
Since $P' \circ (B',\bar{\mu}')^{-1} = P \circ (B,\bar{\mu})^{-1}$, it follows from weak uniqueness of the SDE \eqref{def:qSDE-random} that $P' \circ (B',\bar{\mu}',q')^{-1} = P \circ (B,\bar{\mu},q)^{-1}$, which in turn implies $P' \circ (B',\mu')^{-1} = P \circ (B,\mu)^{-1}$. Using translation invariance of $b$ and $\sigma$, we check that $X' := Y' + q'$ verifies the state equation
\[
dX'_t = [b_0(t,\nu^x_t) + b(t,X'_t,\mu'^x_t,\alpha'_t)]dt + \sigma(t,X'_t,\mu'^x_t,\alpha'_t)dW'_t + \sigma_0(t,\mu'^x_t)dB'_t, \ X'_0=Y'_0.
\]
Hence, $(\Omega',(\F'_t)_{t \in [0,T]},P',B',W',\mu',\alpha',X')$ satisfies (1-5) of Definition \ref{def:wcnsolution}, along with $P' \circ (B',\mu')^{-1} = P \circ (B,\mu)^{-1}$. Using translation invariance of $f$ along with the optimality property (6) of Definition \ref{def:wcnsolution}, we get
\begin{align*}
0 &\le \E^{P}\left[\int_0^Tf(t,X_t,\mu^x_t,\alpha_t)dt + g(X_T,\mu^x_T)\right] - \E^{P'}\left[\int_0^Tf(t,X'_t,\mu'^x_t,\alpha'_t)dt + g(X'_T,\mu'^x_T)\right] \\
	&= \E^{P}\left[\int_0^Tf(t,Y_t,\bar{\mu}^x_t,\alpha_t)dt + g(Y_T,\bar{\mu}^x_T)\right] - \E^{P'}\left[\int_0^Tf(t,Y'_t,\bar{\mu}'^x_t,\alpha'_t)dt + g(Y'_T,\bar{\mu}'^x_T)\right].
\end{align*}
This shows that $(\Omega,(\F_t)_{t \in [0,T]},P,W,\bar{\mu},\alpha,Y)$ is a weak NCN solution.
\end{proof}

A corollary of Theorem \ref{th:mainconverse} lets us relate uniqueness statements for CN solutions and NCN solutions.
We say \emph{uniqueness in law holds for the weak CN solution} if any two weak CN solutions $(\Omega^i,(\F^i_t)_{t \in [0,T]},P^i,B^i,W^i,\mu^i,\alpha^i,X^i)$ induce the same joint law $P^i \circ (B^i,\mu^i)^{-1}$. Note that weak uniqueness along with the fixed point condition (7) implies that also $P^i \circ (B^i,W^i,\mu^i,\alpha^i,X^i)^{-1}$ is the same for each $i=1,2$. Analogously, we say \emph{uniqueness in law holds for the weak NCN solution} if any two weak NCN solutions $(\Omega^i,(\F^i_t)_{t \in [0,T]},P^i,W^i,\bar{\mu}^i,\alpha^i,Y^i)$ induce the same law $P^i \circ (\bar{\mu}^i)^{-1}$.

\begin{corollary} \label{co:uniqueness}
Suppose uniqueness in law holds for weak NCN solutions, and suppose the unique weak NCN solution is in fact a strong NCN solution. Suppose weak uniqueness holds for the SDE \ref{def:qSDE-random}, in the sense of Definition \ref{qSDE-defn}, and suppose that for each (deterministic) $\bar{\mu} \in \P(\C^d)$ that pathwise uniqueness holds for the corresponding SDE \eqref{def:qSDE}. Then uniqueness in law holds for weak CN solutions, and the unique weak CN solution is in fact a strong CN solution.
\end{corollary}
\begin{proof}
Suppose that $(\Omega^i,\F^i_t,P^i,B^i,W^i,\mu^i,\alpha^i,X^i)$, $i=1,2$, are two weak CN solutions. Define
\begin{align*}
q^i_t &:= \int_0^tb_0(s,\mu^{i,x}_s)ds + \int_0^t\sigma_0(s,\mu^{i,x}_s)dB^i_s, \text{ where } \mu^{i,x} := (\mu^i)^x, \\
Y^i &:= X^i - q^i, \quad \bar{\mu}^i := \mu^i(\cdot + (0,0,q^i)).
\end{align*}
By Theorem \ref{th:mainconverse}, $(\Omega^i,\F^i_t,P^i,W^i,\bar{\mu}^i,\alpha^i,Y^i)$ is a weak NCN solution for each $i=1,2$. By assumption, $\bar{\mu}^1$ and $\bar{\mu}^2$ must be equal and almost surely constant. Define $\bar{\mu} \in \P(\X)$ to be their common value. Then, on $\Omega^i$, $(q^i_t)_{t \in [0,T]}$ solves the SDE
\[
dq^i_t = b_0(t,\bar{\mu}^x_t(\cdot - q^i_t))dt + \sigma (t,\bar{\mu}^x_t(\cdot - q^i_t))dB^i_t, \ q^i_0 = 0,
\]
and since pathwise uniqueness holds we conclude that $q^i$ is $B^i$-measurable and that $P^1 \circ (B^1,q^1)^{-1} = P^2 \circ (B^2,q^2)^{-1}$. Since $\mu^i = \bar{\mu}(\cdot - (0,0,q^i))$, we conclude that $\mu^i$ is $B^i$-measurable and that $P^1 \circ (B^1,\mu^1)^{-1} = P^2 \circ (B^2,\mu^2)^{-1}$.
\end{proof}

\begin{remark}
When there is no common noise, i.e. $\sigma_0 \equiv 0$, we still have some flexibility with the drift term $b_0$. Analogs of Theorem \ref{th:main} and Corollary \ref{co:uniqueness} allow us to relate solutions of $NCN(b,\sigma,f,g)$ with solutions of $NCN(b+b_0,\sigma,f,g)$. In particular, existence and uniqueness are equivalent for the two problems, whether we use strong or weak solutions. Many uniqueness results \cite{lasrylionsmfg,ahuja-mfgwellposedness,carmonalacker-probabilisticweakformulation,carmonadelaruelacker-mfgcommonnoise} rely not only on the Lasry-Lions monotonicity condition but also crucially on the independence of the coefficients $b$ and $\sigma$ of the mean field term. Corollary \ref{co:uniqueness} leads to new uniqueness results for (translation invariant) MFGs in which a mean field term enters into the drift.
\end{remark}

\bibliographystyle{amsplain}
\bibliography{translationMFG-bib}

\end{document}